\documentclass[11pt, a4paper]{article}

\usepackage{amsmath, amsthm, amssymb, url, enumerate}
\usepackage[margin=2cm]{geometry}
\usepackage{tikz, tkz-graph, mathrsfs}
\usepackage{authblk}

\usetikzlibrary{arrows}
\newcommand{\vertex}[3]{\node [vertex] (#1) at (#2, #3 * 1.7) {};}
\newcommand{\edge}[2]{\draw (#1) -- (#2);}
\newcommand{\arc}[2]{{\draw[-latex] (#1) edge (#2);}}

\newcommand{\Sing}{\mathrm{Sing}}

\newcommand{\rk}{\mathrm{rk}}
\newcommand{\id}{\mathrm{id}}
\newcommand{\Ima}{\mathrm{Im}}

\newcommand{\cycl}{\mathrm{cycl}}
\newcommand{\fix}{\mathrm{fix}}
\newcommand{\Tour}{\mathrm{Tour}}

\newcommand{\diam}{\mathrm{diam}}
\newcommand{\Acy}{\mathrm{Acyclic}}

\newcommand{\genset}[1]{\langle#1\rangle}

\theoremstyle{plain}
\newtheorem{theorem}{Theorem}[section]
\newtheorem{conjecture}[theorem]{Conjecture}

\newtheorem{lemma}[theorem]{Lemma}

\newtheorem*{claim*}{Claim}
\newtheorem{claim}[theorem]{Claim}

\theoremstyle{definition}
\newtheorem{definition}{Definition}
\newtheorem{problem}{Problem}

\newtheorem{remark}{Remark}

\begin{document}

\title{Lengths of words in transformation semigroups generated by digraphs}
\author[1]{P.J. Cameron}
\author[2]{A. Castillo-Ramirez\footnote{Corresponding author: \texttt{alonso.castillo-ramirez@durham.ac.uk}}}
\author[2]{M. Gadouleau}
\author[1]{J.D. Mitchell}

\affil[1]{School of Mathematics and Statistics, University of St Andrews, St Andrews, Fife KY16 9SS, U.K.}
\affil[2]{School of Engineering and Computing Sciences, Durham University, South Road, Durham DH1 3LE, U.K.}

\maketitle

\begin{abstract}
Given a simple digraph $D$ on $n$ vertices (with $n\ge2$), there is a natural construction of a semigroup $\langle D\rangle$ associated with $D$. For any
edge $(a,b)$ of $D$, let $a\to b$ be the idempotent of defect $1$ mapping $a$ to $b$ and fixing all vertices other than $a$; then define $\langle D\rangle$ to be the semigroup $\langle a\to b:(a,b)\in E(D)\rangle$. For $\alpha \in \genset{D}$, let $\ell(D,\alpha)$ be the minimal length of a word in $E(D)$ expressing $\alpha$. When $D=K_n$ is the complete undirected graph, Howie and Iwahori, independently, obtained a formula to calculate $\ell(K_n,\alpha)$, for any $\alpha \in \genset{K_n} = \Sing_n$; however, no analogous nontrivial results are known when $D \neq K_n$. In this paper, we characterise all simple digraphs $D$ such that either $\ell(D,\alpha)$ is equal to Howie-Iwahori's formula for all $\alpha \in \genset{D}$, or $\ell(D,\alpha) = n - \fix(\alpha)$ for all $\alpha \in \genset{D}$, or $\ell(D,\alpha) = n - \rk(\alpha)$ for all $\alpha \in \genset{D}$. When $D$ is an acyclic digraph and $\alpha \in \genset{D}$, we find a tight upper bound for $\ell(D,\alpha)$. Finally, we study the case when $D$ is a strong tournament (which corresponds to a smallest generating set of idempotents of defect $1$ of $\Sing_n$), and we propose some conjectures. 
\end{abstract}

\section{Introduction}

For any $n \in \mathbb{N}$, $n \geq 2$, let $\Sing_n$  be the semigroup of all singular (i.e. non-invertible) transformations on $[n]:=\left\{1,...,n\right\} $. It is well-known (see \cite{H66}) that $\Sing_n$ is generated by its idempotents of defect $1$ (i.e. the transformations $\alpha \in \Sing_n$ such that $\alpha^2 = \alpha$ and $\rk(\alpha) := \vert \Ima(\alpha) \vert = n-1$). There are exactly $n(n-1)$ such idempotents, and each one of them may be written as $(a \to b)$, for $a,b \in [n]$, $a \ne b$, where, for any $v \in [n]$,
\[ (v) (a \to b) := \begin{cases}
		b &\text{if } v = a,\\
		v &\text{otherwise}.
	\end{cases} \]
Motivated by this notation, we  refer to these idempotents as \emph{arcs}. 

In this paper, we explore the natural connections between simple digraphs on $[n]$ and subsemigroups of $\Sing_n$; explicitly, for any simple digraph $D$ with vertex set $V(D)=[n]$ and edge set $E(D)$, we associate the semigroup
\[ \langle D \rangle := \left\langle (a \to b) \in \Sing_n : (a,b) \in E(D) \right\rangle. \]
We say that a subsemigroup $S$ of $\Sing_n$ is \emph{arc-generated} by a simple digraph $D$ if $S=\langle D \rangle$.

For the rest of the paper, we use the term `digraph' to mean `simple digraph' (i.e. a digraph with no loops or multiple edges). A digraph $D$ is \emph{undirected} if its edge set is a symmetric relation on $V(D)$, and it is \emph{transitive} if its edge set is a transitive relation on $V(D)$. We shall always assume that $D$ is \emph{connected} (i.e. for every pair $u, v \in V(D)$ there is either a path from $u$ to $v$, or a path from $v$ to $u$) because otherwise $\langle D \rangle \cong \genset{D_1} \times \dots \times \genset{D_k}$, where $D_1, \dots, D_k$ are the connected components of $D$. We say that $D$ is \emph{strong} (or \emph{strongly connected}) if for every pair $u,v \in V(D)$, there is a directed path from $u$ to $v$. We say that $D$ is a \emph{tournament} if for every pair $u,v \in V(D)$ we have $(u,v) \in E(D)$ or $(v,u) \in E(D)$, but not both.      

Many famous examples of semigroups are arc-generated. Clearly, by the discussion of the first paragraph, $\Sing_n$ is arc-generated by the complete undirected graph $K_n$. In fact, for $n \geq 3$, $\Sing_n$ is arc-generated by $D$ if and only if $D$ contains a strong tournament (see \cite{H78}). The semigroup of order-preserving transformations $\text{O}_n := \{ \alpha \in \Sing_n : u \le v \Rightarrow u \alpha \le v \alpha \}$ is arc-generated by an undirected path $P_n$ on $[n]$, while the Catalan semigroup $\text{C}_n := \{ \alpha \in \Sing_n : v \le v \alpha, u \le v \Rightarrow u \alpha \le v \alpha \}$ is arc-generated by a directed path $\vec{P}_n$ on $[n]$ (see \cite[Corollary 4.11]{S96}). The semigroup of non-decreasing transformations $\text{OI}_n := \{ \alpha \in \Sing_n : v \le v \alpha \}$ is arc-generated by the transitive tournament $\vec{T}_n$ on $[n]$ (Figure \ref{fig:vecT5} illustrates $\vec{T}_5$).   

\begin{figure}[!h]
\begin{center}
\begin{tikzpicture}
	\node (1) at (0,0) {1};
	\node (2) at (2,0) {2};
	\node (3) at (4,0) {3};
	\node (4) at (6,0) {4};
	\node (5) at (8,0) {5};
	
	\draw[-latex] (1) -- (2); 
	\draw[-latex] (2) -- (3);
	\draw[-latex] (3) -- (4);
	\draw[-latex] (4) -- (5);
	
	\draw[-latex] (1) .. controls(2,1) .. (3);
	\draw[-latex] (1) .. controls(3,2) .. (4);
	\draw[-latex] (1) .. controls(4,3) .. (5);
	
	\draw[-latex] (2) .. controls(4,1) .. (4);
	\draw[-latex] (2) .. controls(5,2) .. (5);
	
	\draw[-latex] (3) .. controls(6,1) .. (5);
\end{tikzpicture}
\end{center}
\caption{$\vec{T}_5$} \label{fig:vecT5}
\end{figure}

Connections between subsemigroups of $\Sing_n$ and digraphs have been studied before (see \cite{S96,YY02,YY06,YY09}). The following definition, which we shall adopt in the following sections, appeared in \cite{YY09}:

\begin{definition}
For a digraph $D$, the \emph{closure} $\bar{D}$ of $D$ is the digraph with vertex set $V(\bar{D}) := V\left( D\right) $ and edge set $E(\bar{D}):=E\left( D\right) \cup \left\{ \left( a,b \right) :\left( b ,a \right) \in E\left( D\right) \text{ is in a directed cycle of }D\right\}$. 
\end{definition}

Say that $D$ is \emph{closed} if $D = \bar{D}$. Observe that $\langle D \rangle = \langle \bar{D} \rangle$ for any digraph $D$.

Recall that the \emph{orbits} of $\alpha \in \Sing_n$ are the connected components of the digraph on $[n]$ with edges $\{ (x, x\alpha) : x \in [n] \}$. In particular, an orbit $\Omega$ of $\alpha$ is called \emph{cyclic} if it is a cycle with at least two vertices. An element $x \in [n]$ is a \emph{fixed point} of $\alpha$ if $x\alpha=x$. Denote by $\cycl(\alpha)$ and $\fix(\alpha)$ the number of cyclic orbits and fixed points of $\alpha$, respectively. Denote by $\ker(\alpha)$ the partition of $[n]$ induced by the \emph{kernel} of $\alpha$ (i.e. the equivalence relation $\{ (x,y) \in [n]^2 : x\alpha = y \alpha \}$).  

We introduce some further notation. For any digraph $D$ and $v \in V(D)$, define the \emph{in-neighbourhood} and the \emph{out-neighbourhood} of $v$ by
\[ N^-(v) := \{ u \in V(D) : (u, v) \in  E(D)\} \text{ and } N^+(v) := \{u \in V(D) : (v,u) \in E(D) \}, \]
respectively. We extend these definitions to any subset  $C \subseteq V(D)$ by letting $N^\epsilon(C) := \bigcup_{c \in C} N^\epsilon(c)$, where $\epsilon \in \{+,- \}$. The \emph{in-degree} and \emph{out-degree} of $v$ are $\deg^-(v):=\vert N^-(v) \vert$ and $\deg^+(v):=\vert N^+(v) \vert$, respectively, while the \emph{degree} of $v$ is $\deg(v) := \vert N^-(v) \cup N^+(v) \vert$. For any two vertices $u,v \in V(D)$, the \emph{$D$-distance} from $u$ to $v$, denoted by $d_D(u,v)$, is the length of a shortest path from $u$ to $v$ in $D$, provided that such a path exists. The \emph{diameter} of $D$ is $\diam(D) := \max \{ d_D(u,v) : u,v \in V(D), \ d_D(u,v) \text{ is defined} \}$.

Let $D$ be any digraph on $[n]$. We are interested in the lengths of transformations of $\genset{D}$ viewed as words in the free monoid $D^* := \{ (a \to b ) : (a,b)\in E(D)\}^*$. Say that a word $\omega \in D^*$ \emph{expresses} (or \emph{evaluates to}) $\alpha \in \genset{D}$ if $\alpha = \omega \phi$, where $\phi : D^* \to \genset{D}$ is the evaluation semigroup morphism. For any $\alpha \in \langle D \rangle$, let $\ell(D,\alpha)$ be the minimum length of a word in $D^*$ expressing $\alpha$. For $r \in [n-1]$, denote
\begin{align*}
	\ell(D,r) &:= \max \left\{  \ell(D,\alpha) : \alpha \in \langle D \rangle, \rk(\alpha) = r \right\},\\
	\ell(D) &:= \max \left\{  \ell(D,\alpha) : \alpha \in \langle D \rangle \right\}.
\end{align*}

The main result in the literature in the study of $\ell(D,\alpha)$ was obtained by Howie and Iwahori, independently, when $D = K_n$.
\begin{theorem}[\cite{H80,I77}]  \label{Howie-Iwahori}
For any $\alpha \in \Sing_n$,
\[	\ell(K_n, \alpha) = n+ \cycl(\alpha) - \fix(\alpha). \]
Therefore, $\ell(K_n, r) = n +  \left\lfloor \frac{1}{2} (r-2) \right\rfloor$, for any $r \in [n-1]$, and $\ell(K_n)  = \ell(K_n, n-1) = \left\lfloor \frac{3}{2} (n-1) \right\rfloor$.
\end{theorem}

In the following sections, we study $\ell(D, \alpha)$, $\ell(D,r)$, and $\ell(D)$, for various classes of digraphs. In Section \ref{sec:short}, we characterise all digraphs $D$ on $[n]$ such that either $\ell(D, \alpha) = n + \cycl(\alpha) - \fix(\alpha)$ for all $\alpha \in \genset{D}$, or $\ell(D, \alpha) =n - \fix(\alpha)$ for all $\alpha \in \genset{D}$, or $\ell(D, \alpha) =n - \rk(\alpha)$ for all $\alpha \in \genset{D}$. In Section \ref{sec:long}, we are interested in the maximal possible length of a transformation in $\genset{D}$ of rank $r$ among all digraphs $D$ on $[n]$ of certain class $\mathcal{C}$; we denote this number by $\ell_{\max}^{\mathcal{C}}(n,r)$. In particular, when $\mathcal{C}$ is the class of acyclic digraphs, we find an explicit formula for $\ell_{\max}^{\mathcal{C}}(n,r)$. When $\mathcal{C}$ is the class of strong tournaments, we find upper and lower bounds for $\ell_{\max}^{\mathcal{C}}(n,r)$ (and for the analogously defined $\ell_{\min}^{\mathcal{C}}(n,r)$), and we provide some conjectures.


\section{Arc-generated semigroups with short words} \label{sec:short}

Let $D$ be a digraph on $[n]$, $n \geq 3$, and $\alpha \in \langle D \rangle$. Theorem \ref{Howie-Iwahori} implies the following three bounds:
\begin{equation}\label{eq:n-rk}
	\ell(D, \alpha) \ge n + \cycl(\alpha) - \fix(\alpha) \ge n - \fix(\alpha) \ge n- \rk(\alpha).
\end{equation}

The lowest bound is always achieved for constant transformations.

\begin{lemma}\label{le:constant}
For any digraph $D$ on $[n]$, if $\alpha \in \langle D \rangle$ has rank $1$, then $\ell(D, \alpha) = n - 1$.
\end{lemma}

\begin{proof}
It is clear that $\ell(D, \alpha) \geq n-1$ because $\alpha$ has $n-1$ non-fixed points. Let $\Ima(\alpha) = \{v_0\} \subseteq [n]$. Note that, for any $v \in [n]$, there is a directed path in $D$ from $v$ to $v_0$ (as otherwise, $\alpha \not \in \langle D \rangle$). For any $d \ge 1$, let
\[ 	C_d := \{ v \in [n] : d_D(v,v_0) = d\}. \]
Clearly, $[n] \setminus \{ v_0 \}= \bigcup_{d=1}^{m} C_d$, where $m := \max_{v\in[n]}\{ d_{D}(v,v_0) \}$ and the union is disjoint. For any $v \in C_d$, let $v'$ be a vertex in $C_{d-1}$ such that $(v \to v') \in D$. For any distinct $v, u \in C_d$ and any choice of $v', u' \in C_{d-1}$, the arcs $(v \to v')$ and $(u \to u')$ commute; hence, we can decompose $\alpha$ as
\[ 	\alpha = \bigcirc_{d=m}^1 \bigcirc_{v \in C_d} (v \to v'), \] 
where the composition of arcs is done from $m$ down to $1$.
\end{proof}

Inspired by the previous lower bounds, in this section we characterise the all connected digraphs $D$ on $[n]$ satisfying the following conditions:
\begin{align}
\forall \alpha \in \langle D \rangle, \ \ell(D, \alpha) & = n + \cycl(\alpha) - \fix(\alpha); \label{bound1} \tag{\textbf{C1}} \\
\forall \alpha \in  \langle D \rangle, \ \ell(D, \alpha) & =  n - \fix(\alpha); \label{bound2} \tag{\textbf{C2}} \\
\forall \alpha \in  \langle D \rangle, \ \ell(D, \alpha) & = n- \rk(\alpha). \label{bound3} \tag{\textbf{C3}}
\end{align}


\subsection{Digraphs satisfying condition (\ref{bound1})}

Theorem \ref{Howie-Iwahori} says that $K_n$ satisfies (\ref{bound1}). In order to characterise all digraphs satisfying (\ref{bound1}), we introduce the following property on a digraph $D$:
\begin{description}
\item[($\star$)] If $d_D(v_0, v_2) = 2$ and $v_0,v_1,v_2$ is a directed path in $D$, then $N^+\left(\{v_1,v_2\}\right) \subseteq \{v_0,v_1,v_2\}$.
\end{description}

We shall study the strong components of digraphs satisfying property ($\star$). We state few observations that we use repeatedly in this section.

\begin{remark} \label{re:star}
Suppose that $D$ satisfies property ($\star$). If $v_0, v_1, v_2$ is a directed path in $D$ and $\deg^+(v_1) >2$, or $\deg^+(v_2) >2$, then $(v_0, v_2) \in E(D)$. Indeed, if $(v_0, v_2) \not \in E(D)$, then $d_D(v_0, v_2) = 2$, so, by property ($\star$), $N^+\left(\{v_1,v_2\}\right) \subseteq \{v_0,v_1,v_2\}$; this contradicts that $\deg^+(v_1) >2$, or $\deg^+(v_2) >2$.
\end{remark}

\begin{remark} \label{re:star2}
Suppose that $D$ satisfies property ($\star$). If $v_0, v_1, v_2$ is a directed path in $D$ and either $v_1$ or $v_2$ have an out-neighbour not in $\{ v_0, v_1, v_2 \}$, then $(v_0,v_1) \in E(D)$.   
\end{remark}

\begin{remark} \label{re:star3}
If $D$ satisfies property ($\star$), then $\diam(D) \leq 2$. Indeed, if $v_0, v_1, \dots, v_k$ is a directed path in $D$ with $d_D (v_0, v_k) = k \geq 3$, then $v_0, v_1, v_2$ is a directed path in $D$ and $v_2$ has an out-neighbour $v_3 \not \in \{ v_0, v_1, v_2 \}$; by Remark \ref{re:star2}, $(v_0,v_2)\in E(D)$, which contradicts that $d_D (v_0, v_k) = k$. 
\end{remark}

Note that digraphs satisfying property ($\star$) are a slight generalisation transitive digraphs.	

Let $D$ be a digraph and let $C_1$ and $C_2$ of be two strong components of $D$. We say that $C_1$ \emph{connects} to $C_2$ if $(v_1, v_2) \in E(D)$ for some $v_1 \in C_1$, $v_2 \in C_2$; similarly, we say that $C_1$ \emph{fully connects} to $C_2$ if $(v_1, v_2) \in E(D)$ for all $v_1 \in C_1$, $v_2 \in C_2$. The strong component $C_1$ is called \emph{terminal} if there is no strong component $C \neq C_1$ of $D$ such that $C_1$ connects to $C$.  

\begin{lemma}\label{le1:bound1}
Let $D$ be a closed digraph satisfying property ($\star$). Then, any strong component of $D$ is either an undirected path $P_3$ or complete. Furthermore, $P_3$ may only appear as a terminal strong component of $D$.
\end{lemma}
\begin{proof}
Let $C$ be a strong component of $D$. Since $D$ is closed, $C$ must be undirected. The lemma is clear if $\vert C \vert \leq 3$, so assume that $\vert C \vert \geq 4$. We have two cases:
\begin{description}
\item[Case 1:] Every vertex in $C$ has degree at most $2$. Then $C$ is a path or a cycle. Since $\vert C \vert \geq 4$ and $\diam(D) \leq 2$, then $C$ is a cycle of length $4$ or $5$; however, these cycles do no satisfy property ($\star$).

\item[Case 2:] There exists a vertex $  a \in C$ of degree $3$ or more. Any two neighbours of $a$ are adjacent: indeed, for any $u,v \in N(a)$, $u,a,v$ is a path and $\deg^+(a) > 2$, so $(u,v) \in E(D)$ by Remark \ref{re:star}. Hence, the neighbourhood of $a$ is complete and and every neighbour of $a$ has degree $3$ or more. Applying this rule recursively, we obtain that every vertex in $C$ has degree $3$ or more, and the neighbourhood of every vertex is complete. Therefore, $C$ is complete because $\diam(D) \leq 2$.
\end{description}

Finally, if $P_3$ is a strong component of $D$, there cannot be any edge coming out of it because of property ($\star$), so it must be a terminal component.
\end{proof}

\begin{lemma}\label{le2:bound1}
Let $D$ be a closed digraph satisfying property ($\star$). Let $C_1$ and $C_2$ be strong components of $D$, and suppose that $C_1$ connects to $C_2$.
\begin{description}
\item[(i)] If $C_2$ is nonterminal, then $C_1$ fully connects to $C_2$.
\item[(ii)] Let $|C_2|  = 1$. If either $|C_1| \neq 2$, or the vertex in $C_1$ that connects to $C_2$ has out-degree at least $3$, then $C_1$ fully connects to $C_2$.
\item[(iii)] Let $|C_2|  = 2$. If not all vertices in $C_1$ connect to the same vertex in $C_2$, then $C_1$ fully connects to $C_2$.
\item[(iv)] If $|C_2| \ge 3$, then $C_1$ fully connects to $C_2$.
\end{description}
\end{lemma}
\begin{proof}
Recall that $C_1$ and $C_2$ are undirected because $D$ is closed. If $|C_1| = 1$ and $|C_2| = 1$, clearly $C_1$ fully connects to $C_2$. Henceforth, we assume $|C_1| \geq 2$ or $|C_2| \geq 2$. Let $c_1 \in C_1$ and $c_2 \in C_2$ be such that $(c_1,c_2) \in E(D)$. As $C_1$ is a nonterminal, Lemma \ref{le1:bound1} implies that $C_1$ is complete.
\begin{description}
\item[(i)] As $C_2$ is nonterminal, there exists $d \in D \setminus (C_1 \cup C_2)$ such that $(c_2,d) \in E(D)$. Suppose that $|C_1| \geq 2$. Then, for any $c'_1 \in C_1 \setminus \{ c_1 \}$, $c'_1, c_1, c_2$ is a directed path in $D$ with $d \in N^{+}( c_2 )$, so Remark \ref{re:star2} implies $(c'_1, c_2) \in E(D)$. Suppose now that $|C_2| \geq 2$. Then, for any $c'_2 \in C_2 \setminus \{ c_2\}$, $c_1, c_2, c'_2$ is a directed path in $D$ with $d \in N^{+}(c_2)$, so again $(c_1, c'_2) \in E(D)$. Therefore, $C_1$ fully connects to $C_2$.    

\item[(ii)] Suppose that $|C_1| \geq 2$. If $\vert C_1 \vert >2$, then $\deg^+(c_1) > 2$, because $C_1$ is complete. Thus, for each $c'_1 \in C_1 \setminus \{ c_1 \}$, $c'_1, c_1, c_2$ is a directed path in $D$ with $\deg^+(c_1) > 2$, so $(c'_1, c_2) \in E(D)$ by Remark \ref{re:star}. As $|C_2|  = 1$, this shows that $C_1$ fully connects to $C_2$. 

\item[(iii)] Let $C_2 = \{ c_2, c'_2 \}$ and let $c'_1 \in C_1 \setminus \{ c_1 \}$ be such that $(c'_1, c'_2) \in E(D)$. For any $b ,d\in C_1 $, $b \neq c_1$, $d \neq c'_1$, both $b, c_1, c_2$ and $d, c'_1, c'_2$ are directed paths in $D$ with $c'_2 \in N^+(c_2)$ and $c_2 \in N^+(c'_2)$; hence, $(b,c_2) , (d, c'_2) \in E(D)$ by Remark \ref{re:star2}.    

\item[(iv)] Suppose that $C_2 = P_3$. Say $C_2 = \{ c_2, c'_2, c''_2\}$ with either $d_{D}(c_2, c''_2)=2$ or $d_{D}(c'_2, c''_2)=2$. In any case, $c_1, c_2, c'_2$ is a directed path in $D$ with $c''_2 \in N^+ (\{c_2, c'_2 \})$, so $(c_1, c'_2) \in E(D)$ by Remark \ref{re:star2}; now, $c_1, c'_2, c''_2$ is a directed path in $D$ with $c_2 \in N^+ (\{c'_2, c''_2 \})$, so $(c_1, c''_2) \in E(D)$. Hence, $c_1$ is connected to all vertices of $C_2$. As $C_1$ is complete, a similar argument shows that every $c'_1 \in C_1 \setminus \{ c_1\}$ connects to every vertex in $C_2$.

Suppose now that $C_2 = K_m$ for $m \ge 3$. By a similar reasoning as the previous paragraph, we show that $(c_1, v) \in E(D)$ for all $v \in C_2$. Now, for any $c'_1 \in C_1 \setminus \{ c_1 \}$, $v \in C_2$, $c'_1, c_1, v$ is a directed path in $D$ so $(c'_1, v) \in E(D)$ by Remark \ref{re:star2}.
\end{description}
\end{proof}

\begin{lemma}\label{le3:bound1}
Let $D$ be a closed digraph satisfying property ($\star$). Let $C_i$, $i=1,2,3$, be strong components of $D$, and suppose that $C_1$ connects to $C_2$ and $C_2$ connects to $C_3$.  If $C_1$ does not connect to $C_3$, then $|C_2| = |C_3| = 1$, $C_3$ is terminal in $D$, and $C_2$ is terminal in $D \setminus C_3$.
\end{lemma}
\begin{proof}
By Lemma \ref{le2:bound1} \textbf{(i)}, $C_1$ fully connects to $C_2$. Assume that $C_1$ does not connect to $C_3$. Let $c_i \in C_i$, $i=1,2,3$, be such that $(c_1, c_2), (c_2, c_3) \in E(D)$. If $C_2$ has a vertex different from $c_2$, Remark \ref{re:star2} ensures that $(c_1, c_3) \in  E(D)$, which contradicts our hypothesis. Then $\vert C_2 \vert =1$. The same argument applies if $C_3$ has a vertex different from $c_3$, so $\vert C_3 \vert =1$. Finally, Remark \ref{re:star2} applied to the path $c_1, c_2, c_3$ also implies that $C_3$ is terminal in $D$ and $C_2$ is terminal in $D \setminus C_3$.
\end{proof}

The following result characterises all digraphs satisfying condition (\ref{bound1}).

\begin{theorem} \label{th:l=g}
Let $D$ be a connected digraph on $[n]$. The following are equivalent:
\begin{description}
	\item[(i)] \label{it:l=g1} For all $\alpha \in \langle D \rangle$, $\ell(D, \alpha) = n + \cycl(\alpha) - \fix(\alpha)$.
	
	\item[(ii)] \label{it:l=g2} $D$ is closed satisfying property ($\star$).
\end{description}
\end{theorem}

\begin{proof}
In order to simplify notation, denote 
\[ g(\alpha) := n + \cycl(\alpha) - \fix(\alpha). \]
First, we show that \textbf{(i)} implies \textbf{(ii)}. Suppose $\ell(D,\alpha) = g(\alpha)$ for all $\alpha \in \genset{D}$. We use the one-line notation for transformations: $\alpha = (1)\alpha \ (2)\alpha \ \dots \ (k)\alpha$, where $x=(x)\alpha$ for all $x >k$, $x \in[n]$. Clearly, if $D$ is not closed, there exists an arc $\alpha \in \genset{D} \backslash D$, so $1 < \ell(D, \alpha) \neq g(\alpha) = 1$. In order to prove that property ($\star$) holds, let $1,2,3$ be a shortest path in $D$. If $(2 \to v) \in \genset{D}$, for some $v \in [n]\setminus\{1,2,3\}$, then $\alpha = 3v3v \in \genset{D}$, but $g(\alpha) = 2 \neq \ell(D,\alpha) = 3$. If $(3 \to v) \in \genset{D}$, then $\alpha = 3vvv \in \genset{D}$, but $g(\alpha) = 3 \neq \ell(D,\alpha) = 4$. Therefore, $N^+(\{2,3 \}) \subseteq \{ 1,2,3 \}$, and ($\star$) holds. \\ 

Conversely, we show that \textbf{(ii)} implies \textbf{(i)}. Let $\alpha \in \genset{D}$. We remark that any cycle of $\alpha$ belongs to a strong component of $D$. 

\begin{claim} \label{claim:alpha_SC}
Let $C$ be a strong component of $D$. Then either $\alpha$ fixes all vertices of $C$ or $|(C \alpha) \cap C| < |C|$.
\end{claim}
\begin{proof}
Suppose that $\alpha \vert_C$, the restriction of $\alpha$ to $C$, is non-trivial and $|(C \alpha) \cap C| = |C|$. Then $\alpha \vert_C$ is a permutation of $C$. Let $u \in C$ and suppose that $(u \to v)$  is the first arc moving $u$ in a word expressing $\alpha$ in $D^*$. If $v \in C$, we have $u \alpha = v \alpha$, which contradicts that $\alpha \vert_C$ is a permutation. If $v \in C'$ for some other strong component $C'$ of $D$, then $u \alpha \notin C$ which again contradicts our assumption.
\end{proof}

\begin{claim} \label{claim:d=2}
Let $u,v \in [n]$ be such that $u \alpha = v$. If $d_D(u, v) = 2$, then: 
\begin{enumerate}
	\item $v$ is in a terminal component of $D$.
	
	\item There is a path $u,w,v$ of length $2$ in $D$ such that $w \alpha = v \alpha = v$; for any other path $u,x,v$ of length $2$ in $D$, we have $x \alpha \in \{x, v\}$.
\end{enumerate}
\end{claim}
\begin{proof}
Let $C_1$ and $C_2$ be strong components of $D$ such that $u \in C_1$ and $v \in C_2$. We analyse the four possible cases in which $d_D(u,v) = 2$. In the first three cases, we use the fact that $\genset{P_3} \cong \mathrm{O}_3$, hence we can order $u < w < v$ and $\alpha$ is an increasing transformation of the ordered set $\{u,w,v\}$; thus $u \alpha = w \alpha = v \alpha = v$.
\begin{description}
\item[Case 1:] $C_1 = C_2 $. By Lemma \ref{le1:bound1}, $C_1 \cong P_3$ and it is a terminal component. Therefore, \emph{2.} holds as there is a unique path from $u$ to $v$.

\item[Case 2:] $C_1$ connects to $C_2$ and $|C_2| \ne 2$. As $d_D(u, v) = 2$, $C_1$ does not fully connect $C_2$, so, by Lemma \ref{le2:bound1}, $|C_2| = 1$, $C_2$ is terminal, $|C_1| = 2$, and the vertex $w \in C_1$ connecting to $C_2=\{ v\}$ has out-degree $2$. Then, by property ($\star$), $u,w,v$ is the unique path from $u$ to $v$. 

\item[Case 3:] $C_1$ connects to $C_2$ and $|C_2| = 2$. As $d_D(u, v) = 2$, $C_1$ does not fully connect $C_2$, so, by Lemma \ref{le2:bound1}, $C_2$ is terminal and $u,w,v$ is the unique path of length two from $u$ to $v$, where $w$ is the other vertex of $C_2$. 

\item[Case 4:] $C_1$ does not connect to $C_2$. Since $d_D(u, v) = 2$, there exist strong components $C^{(1)}, \dots, C^{(k)}$ such that $C_1$ connects to $C^{(i)}$ and $C^{(i)}$ connects to $C_2$, for all $1 \le i \le k$. By Lemma \ref{le3:bound1}, $C^{(i)} = \{ x_i \}$, $C_2 = \{v\}$ is terminal and $N^+(x_i) = \{v\}$ for all $i$. Thus $u, x_i, v$ are the only paths of length two from $u$ to $v$; in particular, $x_i \alpha \in \{x_i, v\}$ for all $x_i$. As $u \alpha = v$, there must exist $1 \le j \le k$ such that $w := x_j$ is mapped to $v$.
\end{description}
\end{proof}

Now we produce a word $\omega \in D^*$ expressing $\alpha$ of length $g(\alpha)$. Define 
\[ U := \{ u \in D : d_D(u, u\alpha) = 2 \}. \]
For every $u \in U$, let $u^\prime$ be a vertex in $D$ such that $u, u^\prime , u \alpha$ is a path and $u^\prime \alpha = u \alpha$. The existence of $u'$ is guaranteed by Claim \ref{claim:d=2}. Define a word $\omega_0 \in D^*$ by
\[ 	\omega_0 := \bigcirc_{u \in U} (u \to u^\prime) (u^\prime \to u \alpha ). \] 

Sort the strong components of $D$ in topological order: $C_1, \dots, C_k$, i.e. for $i \neq j$, $C_i$ connects to $C_j$ only if $j > i$. For each $1 \leq i \leq k$, define
\[ 	S_i := \{v \in C_i \setminus (U \cup U^\prime) : v \alpha \in C_i\}, \]
where $U^\prime := \{ u^\prime : u \in U \}$, and consider the transformation $\beta_i : C_i \to C_i$ defined by
\[ 	x \beta_i = \begin{cases}
		x \alpha &\text{if } x \in S_i\\
		x &\text{otherwise}.
	\end{cases} \]
If $|C_i| \le 2$ or $C_i \cong P_3$, then $\cycl(\beta_i) = 0$ and $\beta_i$ can be computed with $\vert C_i \vert - \fix(\beta_i)$ arcs. Otherwise, $C_i$ is a complete undirected graph. If $\beta_i \in \Sing(C_i)$, then by Theorem \ref{Howie-Iwahori}, there is a word $\omega_i \in C_i^* \subseteq D^*$ of length $\vert C_i \vert + \cycl(\beta_i) - \fix(\beta_i)$ expressing $\beta_i$. Suppose now that $\beta_i$ is a non-identity permutation of $C_i$. By Claim \ref{claim:alpha_SC}, $\alpha$ does not permute $C_i$  and there exists $h_i \in C_i \setminus (C_i \alpha)$. Note that $h_i \in C_i \setminus S_i$. Define $\hat{\beta_i} \in \Sing(C_i)$ by
\[ 	x \hat{\beta_i} = \begin{cases}
		x \alpha &\text{if $x\in S_i$}\\
		a_i & \text{if } x = h_i \\
		x &\text{otherwise},
	\end{cases}\]
where $a_i$ is any vertex in $S_i$. Then $\alpha \vert_{S_i} = \hat{\beta}\vert_{S_i}$. Again by Theorem \ref{Howie-Iwahori}, there is a word $\omega_i \in C_i^* \subseteq D^*$ of length $\vert C_i \vert + \cycl(\hat{\beta}_i) - \fix(\hat{\beta}) =\vert C_i \vert + \cycl(\beta_i) - \fix(\beta_i) $ expressing $\hat{\beta}_i$.

The following word maps all the vertices in $[n] \setminus (U \cup U^\prime \cup C_i)$ that have image in $C_i$:
\[ 	\omega_i' = \bigcirc \left\{ (a \to a\alpha) : a \in [n] \setminus (U \cup U^\prime \cup C_i), a\alpha \in C_i \right\}. \]

Finally, let
\[ 	\omega := \omega_0 \omega_k \omega_k' \dots \omega_1 \omega_1' \in D^*. \]
It is easy to check that $\omega$ indeed expresses $\alpha$. Since $\sum_{i=1}^k \fix(\beta_i) = \fix(\alpha) + \sum_{i=1}^k \vert C_i \setminus  S_i \vert$ and $\sum_{i=1}^k \ell(\omega^\prime_i) = \sum_{i=1}^k \vert C_i \setminus ( U \cup U^\prime \cup S_i) \vert$, we have
\[ 	\ell(\omega) = 2 \vert U \vert + \sum_{i=1}^k (\ell(\omega_i) + \ell(\omega_i^\prime)) =  n + \sum_{i=1}^k \cycl(\beta_i) -  \fix(\alpha) =  g(\alpha). \]
\end{proof}


\subsection{Digraphs satisfying condition (\ref{bound2})}

The characterisation of connected digraphs satisfying condition (\ref{bound2}) is based on the classification of connected digraphs $D$ such that $\cycl(\alpha) = 0$, for all $\alpha \in \genset{D}$. 

For $k \ge 3$, let $\Theta_k$ be the directed cycle of length $k$. Consider the digraphs $\Gamma_1, \ \Gamma_2, \ \Gamma_3$ and $\Gamma_4$ as illustrated below:

 
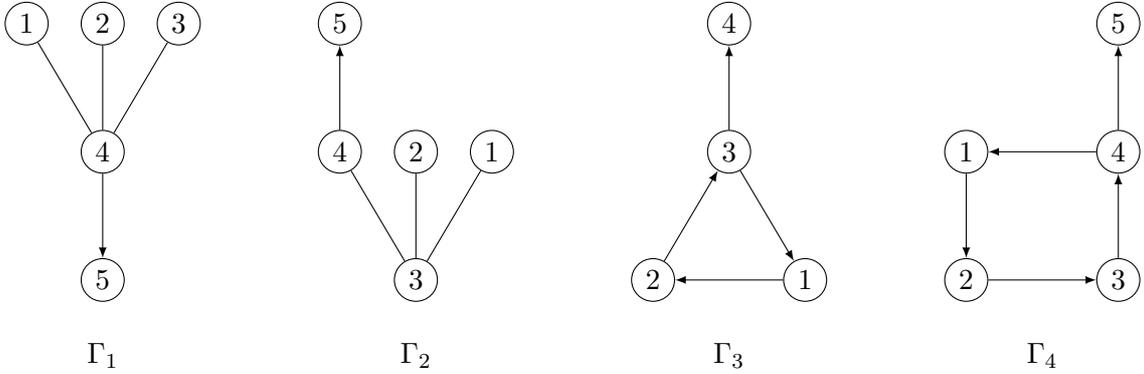
\[ \begin{tikzpicture}[vertex/.style={circle, draw, fill=none, inner sep=0.2cm}]
    \vertex{1}{0}{2}
    \node at (0,3.4) {1}; 
    \vertex{2}{1}{2}
    \node at (1,3.4) {2};  
    \vertex{3}{2}{2}
    \node at (2,3.4) {3};   
   \vertex{4}{1}{1} 
   \node at (1,1.7) {4};  
   \vertex{5}{1}{0}
   \node at (1,0) {5};   

    \edge{1}{4}
    \edge{2}{4}
    \edge{3}{4}
    \arc{4}{5}

\node at (1,-1) {$\Gamma_1$};
  \end{tikzpicture}
  \qquad\qquad
  \begin{tikzpicture}[vertex/.style={circle, draw, fill=none, inner sep=0.2cm}]
\node at (1,-1) {$\Gamma_2$};

    \vertex{1}{1}{1} 
    \node at (1,1.7) {2};   
    \vertex{2}{2}{1}
    \node at (2,1.7) {1};    
   \vertex{3}{1}{0}
    \node at (1,0) {3};   
    \vertex{4}{0}{1} 
     \node at (0,1.7) {4};    
   \vertex{5}{0}{2}
       \node at (0,3.4) {5};    

    \edge{1}{3}
    \edge{2}{3}
    \edge{3}{4}
    \arc{4}{5}
  \end{tikzpicture} 
    \qquad\qquad
  \begin{tikzpicture}[vertex/.style={circle, draw, fill=none, inner sep=0.2cm}]
\node at (1,-1) {$\Gamma_3$};

    \vertex{1}{2}{0}
    \node at (2,0) {1}; 
    \vertex{2}{0}{0}
        \node at (0,0) {2};  
    \vertex{3}{1}{1} 
        \node at (1,1.7) {3}; 
    \vertex{4}{1}{2} 
        \node at (1,3.4) {4}; 
    
    \arc{1}{2}
    \arc{2}{3}
    \arc{3}{1}
    \arc{3}{4}
  \end{tikzpicture}
  \qquad\qquad
  \begin{tikzpicture}[vertex/.style={circle, draw, fill=none, inner sep=0.2cm}]
\node at (1,-1) {$\Gamma_4$};

    \vertex{1}{0}{1}
    \node at (0,1.7) {1};   
    \vertex{2}{0}{0}
     \node at (0,0) {2};  
    \vertex{3}{2}{0} 
    \node at (2,0) {3}; 
    \vertex{4}{2}{1}
    \node at (2,1.7) {4}; 
   \vertex{5}{2}{2}
   \node at (2,3.4) {5}; 
    
    \arc{1}{2}
    \arc{2}{3}
    \arc{3}{4}
    \arc{4}{1}
	\arc{4}{5}
  \end{tikzpicture} \]

\begin{lemma} \label{prop:acyclic_alpha}
Let $D$ be a connected digraph on $[n]$. The following are equivalent:
\begin{description}
	\item[(i)] For all $\alpha \in \langle D \rangle$, $\cycl(\alpha) = 0$.
	
	\item[(ii)] $D$ has no subdigraph isomorphic to $\Gamma_1$, $\Gamma_2$, $\Gamma_3$, $\Gamma_4$, or $\Theta_k$, for all $k \ge 5$.
\end{description}
\end{lemma}	
\begin{proof}
In order to prove that \textbf{(i)} implies \textbf{(ii)}, we show that if $\Gamma$ is equal to $\Gamma_i$ or $\Theta_k$, for $i\in [4]$, $k \ge 5$, then there exists $\alpha \in \genset{\Gamma}$ such that $\cycl(\alpha) \neq 0$.  
\begin{itemize}
\item If $\Gamma = \Gamma_1$, take
\[ 	\alpha := (3 \to 4) (4 \to 5) (1 \to 4) (4 \to 3) (2 \to 4) (4 \to 1) (3 \to 4) (4 \to 2) = 21555. \]

\item If $\Gamma = \Gamma_2$, take
\[ 	\alpha := (3 \to 4) (4 \to 5) (1 \to 3) (3 \to 4) (2 \to 3) (3 \to 1) (4 \to 3) (3 \to 2) = 21555. \]

\item If $\Gamma = \Gamma_3$, take
\[ 	\alpha := (3 \to 4) (2 \to 3) (1 \to 2) (3 \to 1) = 2144. \]

\item If $\Gamma = \Gamma_4$, take
\[	\alpha = (3 \to 4) (4 \to 5) (2 \to 3) (3 \to 4) (1 \to 2) (4 \to 1) = 21555. \]

\item Assume $\Gamma = \Theta_k$ for $k \ge 5$. Consider the following transformation of $[k]$:
	\[ (u \Rightarrow v) := (u \to u_1) \dots (u_{d-1} \to v), \]
where $u,u_1, \dots, u_{d-1}, v$ is the unique path from $u$ to $v$ on the cycle $\Theta_k$. Take
\[ 	\alpha := (1 \Rightarrow k-3) (k \Rightarrow k-4) (k-1 \Rightarrow 1) (k-2 \Rightarrow k) (k-3 \Rightarrow k-1) (k-4 \Rightarrow k-2). \]
Then, $\alpha = (k-1)(k-1) \dots (k-1) \ k \ 1 \ (k-2) $, where $(k-1)$ appears $k-3$ times, has the cyclic component $(k-2, k)$.
\end{itemize}

Conversely, assume that $D$ satisfies \textbf{(ii)}. If $n \leq 3$, it is clear that $\cycl(\alpha) = 0$, for all $\alpha \in \genset{D}$, so suppose $n \geq 4$. We first obtain some key properties about the strong components of $\bar{D}$.

\begin{claim} \label{claim:SC_path}
Any strong component of $\bar{D}$ is an undirected path, an undirected cycle of length 3 or 4, or a claw $K_{3,1}$ (i.e. a bipartite undirected graph on $[4] = [3] \cup \{4\}$). Moreover, if a strong component of $D$ is not an undirected path, then it is terminal.
\end{claim}
\begin{proof}
Let $C$ be a strong component of $\bar{D}$. Clearly, $C$ is undirected and, by \textbf{(ii)}, it cannot contain a cycle of length at least 5. If $C$ has a cycle of length $3$ or $4$, then the whole of $C$ must be that cycle and $C$ is terminal (otherwise, it would contain $\Gamma_3$ or $\Gamma_4$, respectively). If $C$ has no cycle of length $3$ and $4$, then $C$ is a tree. It can only be a path or $K_{3,1}$, for otherwise it would contain $\Gamma_1$ or $\Gamma_2$; clearly, $K_{3,1}$ may only appear as a terminal component.
\end{proof}

Suppose there is $\alpha \in \genset{D}$ that has a cyclic orbit (so $\cycl(\alpha) \neq 0$). This cyclic orbit must be contained in a strong component $C$ of $\bar{D}$, and Claim \ref{claim:SC_path} implies that $C \cong \Gamma$, where $\Gamma \in \{ K_{3,1}, \bar{\Theta}_s, P_r : s \in \{3,4\}, r \in \mathbb{N} \}$. If $\Gamma = K_{3,1}$ or $\Gamma = \bar{\Theta}_s$, then $C$ is a terminal component, so $\alpha$ acts on $C$ as some transformation $\beta \in \genset{ \Gamma}$; however, it is easy to check that no transformation in $\genset{ \Gamma}$ has a cyclic orbit. If $\Gamma = P_r$, for some $r$, then $\alpha$ acts on $C$ as a partial transformation $\beta$ of $P_r$. Since $\genset{P_r} = \mathrm{O}_r$, $\beta$ has no cyclic orbit.
\end{proof}

We introduce a new property of a connected digraph $D$:
\begin{description}
	   \item[($\star\star$)] \label{it:SC_2} For every strong component $C$ of $D$, $|C| \le 2$ if $C$ is nonterminal, and $|C| \le 3$ if $C$ is terminal.
\end{description}

\begin{lemma}\label{le:bound2}
Let $D$ be a closed connected digraph on $[n]$ satisfying property ($\star$). The following are equivalent:
\begin{description}
\item[(i)] $D$ satisfies property ($\star\star$).
\item[(ii)] $D$ has no subdigraph isomorphic to $\Gamma_1$, $\Gamma_2$, $\Gamma_3$, $\Gamma_4$, or $\Theta_k$, for some $k \ge 5$.
\end{description}
\end{lemma}
\begin{proof}
If \textbf{(i)} holds, it is easy to check that $D$ does not contain any subdigraphs isomorphic to $\Gamma_1$, $\Gamma_2$, $\Gamma_3$, $\Gamma_4$, or $\Theta_k$ for some $k \ge 5$.

Conversely, suppose that \textbf{(ii)} holds. Let $C$ be a strong component of $D$. If $C$ is non-terminal, Lemma \ref{le1:bound1} implies that $C$ is complete; hence, $|C| \le 2$ as otherwise $D$ would contain $\Gamma_4$ as a subdigraph. If $C$ is terminal, Lemma \ref{le1:bound1} implies that $C$ is complete or $P_3$; hence, $|C| \le 3$ as otherwise $D$ would contain $\Gamma_3$ as a subdigraph.
\end{proof}

\begin{theorem} \label{th:l=n-fix}
Let $D$ be a connected digraph on $[n]$. The following are equivalent:
\begin{description}
	\item[(i)] For all $\alpha \in \langle D \rangle$, $\ell(D, \alpha) = n - \fix(\alpha)$.
	
	\item[(ii)] $D$ is closed satisfying properties ($\star$) and ($\star\star$). 
\end{description}
\end{theorem}
\begin{proof}
Clearly, $D$ satisfies \textbf{(i)} if and only if it satisfies condition (\ref{bound1}) and $\cycl(\alpha) = 0$, for all $\alpha \in \langle D \rangle$. By Theorem \ref{th:l=g}, Lemma \ref{prop:acyclic_alpha} and Lemma \ref{le:bound2}, $D$ satisfies \textbf{(i)} if and only if  $D$ satisfies \textbf{(ii)}.
\end{proof}


\subsection{Digraphs satisfying condition (\ref{bound3})}

The following result characterises digraphs satisfying condition (\ref{bound3}).

\begin{theorem}\label{thm-bands}
Let $D$ be a connected digraph on $[n]$. The following are equivalent:
\begin{description}
	\item[(i)] For every $\alpha \in \langle D \rangle$, $\ell(D, \alpha) = n - \rk(\alpha)$.
	
	\item[(ii)] $\langle D \rangle$ is a band, i.e. every $\alpha \in \langle D \rangle$ is idempotent.
	
	\item[(iii)] Either $n=2$ and $D \cong K_2$, or there exists a bipartition $V_1 \cup V_2$ of $[n]$ such that $(i_1,i_2) \in E(D)$ only if $i_1 \in V_1$, $i_2 \in V_2$.
\end{description}
\end{theorem}

\begin{proof}
Clearly \textbf{(i)} implies \textbf{(ii)}: if $\ell(D,\alpha) = n - \rk(\alpha)$, then $\rk(\alpha) = \fix(\alpha)$ by inequality (\ref{eq:n-rk}), so $\alpha$ is idempotent. 

Now we prove that \textbf{(ii)} implies \textbf{(iii)}. If there exist $u, v, w \in [n]$ pairwise distinct such that $(u,v), (v,w) \in E(D)$ , then $\alpha = (v \to w) (u \to v)$ is not an idempotent. Therefore, for $n \ge 3$, if every $\alpha \in \langle D \rangle$ is idempotent, then a vertex in $D$ either has in-degree zero or out-degree zero: this corresponds to the bipartition of $[n]$ into $V_1$ and $V_2$.

We finally prove that \textbf{(iii)} implies \textbf{(i)}. Let $n \ge 3$ and suppose that there exists a bipartition $V_1 \cup V_2$ of $[n]$ such that $(i_1,i_2) \in E(D)$ only if $i_1 \in V_1$, $i_2 \in V_2$. Then for any $\alpha \in \langle D \rangle$, all elements of $V_2$ are fixed by $\alpha$ and $i_1 \alpha \in \{i_1\} \cup N^+(i_1)$ for any $i_1 \in V_1$. In particular, any non-fixed point of $\alpha$ is mapped to a fixed point, so $r:=\rk(\alpha) = \fix(\alpha)$. Let $J := \{v_1, \dots, v_{n-r}\} \subseteq V_1$ be the set of non-fixed points of $\alpha$; therefore
\[ 	\alpha = (v_1 \to v_1 \alpha) \dots (v_{n-r} \to v_{n-r} \alpha), \]
where each one of the $n-r$ arcs above belongs to $\genset{D}$. The result follows by inequality (\ref{eq:n-rk}).
\end{proof}


\section{Arc-generated semigroups with long words} \label{sec:long}

Fix $n \geq 2$. In this section, we consider digraphs $D$ that maximise $\ell(D,r)$ and $\ell(D)$. For $r \in [n-1]$, define
\begin{align*}
	\ell_{\max}(n,r) & := \max \left\{ \ell(D, r) :  V(D) = [n]  \right\}, \\
	\ell_{\max}(n) &:= \max \left\{ \ell(D) : V(D) = [n]  \right\}.
\end{align*}

\begin{problem}
Is $\ell_{\max}(n)$ upper bounded by a polynomial in $n$?
\end{problem}

\begin{table}[!h]
\begin{center}
\begin{tabular}{c|c| c|c|c|c}
  $n \backslash r$ & $1$ & $2$ & $3$ & $4$ & $5$ \\
  \hline 
  $2$  		& $1$ & & & &  \\
  $3$               & $2$ & $6$ &   &   &    \\
  $4$                & $3$ & $11$ & $13$ &   &    \\
  $5$                & $4$ & $18$ & $24$ & $33$ &    \\
  $6$ 		& $5$ & $26$ & $42$ & $51$ & $66$
\end{tabular}
\caption{First values of $\ell_{\max}(n,r)$} \label{table:lmax}
\end{center}
\end{table}

The first few values of $\ell_{\max}(n,r)$, calculated with the GAP package \emph{Semigroups} \cite{M15}, are given in Table \ref{table:lmax}. By Lemma \ref{le:constant}, $\ell_{\max}(n,1) =  n-1$ for all $n \geq 2$; henceforth, we shall always assume that $n \geq 3$ and $r \in [n-1] \setminus \{ 1\}$.

In the following sections, we restrict the class of digraphs that we consider in the definition of $\ell_{\max}(n,r)$ and $\ell_{\max}(n)$ to two important cases: acyclic digraphs and strong tournaments.  


\subsection{Acyclic digraphs}

For any $n \geq 3$, let $\Acy_n$ be the set of all acyclic digraphs on $[n]$, and, for any $r \in [n-1]$, define
\begin{align*}
	\ell_{\max}^{\Acy}(n,r) &:= \max \left\{ \ell(A, r) :  A \in \Acy_n \right\},\\
	\ell_{\max}^{\Acy}(n) &:= \max \left\{ \ell(A) : A \in \Acy_n  \right\}.
\end{align*}
Without loss of generality, we assume that any acyclic digraph $A$ on $[n]$ is topologically sorted, i.e. $(u, v) \in E(A)$ only if $v > u$. 

In this section, we establish the following theorem.

\begin{theorem} \label{th:lAcy}
For any $n\geq 3$ and $r \in [n-1] \setminus \{ 1\}$,
\begin{align*}
\ell_{\max}^{\Acy}(n,r) &= \frac{(n-r)(n+r-3)}{2} + 1, \\
\ell_{\max}^{\Acy}(n) & = \ell_{\max}^{\Acy}(n,2) = \frac{1}{2} (n^2 -3n + 4). 
\end{align*}
\end{theorem}

First of all, we settle the case $r = n-1$, for which we have a finer result.

\begin{lemma}
Let $n \geq 3$ and $A \in \Acy_n$. Then, $\ell(A,n-1)$ is equal to the length of a longest path in $A$. Therefore, 
\[ 	\ell_{\max}^{\Acy}(n,n-1) = n-1. \]
\end{lemma}
\begin{proof}
Let $v_1, \dots, v_{l+1}$ be a longest path in $A$. Then $\alpha \in \langle A \rangle$ defined by
\[ 	v \alpha := \begin{cases}
	v_{i+1} &\text{if } v = v_i, \ i \in [l], \\
	v &\text{otherwise},
	\end{cases} \]
has rank $n-1$ and requires at least $l$ arcs, since it moves $l$ vertices.

Conversely, let $\alpha \in A$ be a transformation of rank $n-1$, and consider a word expressing $\alpha$ in $A^*$:
\[ 	\alpha = (u_1 \to v_1) (u_2 \to v_2) \dots (u_s \to v_s). \]
Since $\alpha$ has rank $n-1$, we must have $v_2 = u_1$ and by induction $v_i = u_{i-1}$ for $2 \le i \le s$. As $A$ is acyclic, $u_s, u_{s-1}, \dots, u_1, v_1$ forms a path in $A$, so $s \le l$.
\end{proof}

The following lemma shows that the formula of Theorem \ref{th:lAcy} is an upper bound for $\ell_{\max}^{\Acy}(n,r)$.

\begin{lemma}
For any $n\geq 3$ and $r \in [n-1] \setminus \{ 1\}$,
\[ 	\ell_{\max}^{\Acy}(n,r) \leq \frac{(n-r)(n+r-3)}{2} + 1. \]
\end{lemma}
\begin{proof}
Let $A$ be an acyclic digraph on $[n]$, let $\alpha \in \genset{A}$ be a transformation of rank $r \geq 2$, and let $L \subset V(A)$ be the set of terminal vertices of $A$. For any $u,v \in [n]$, denote the length of a longest path from $u$ to $v$ in $A$ as $\psi_A(u,v)$. 

\begin{claim}\label{lem:l_A}
$\ell(A, \alpha) \le \sum_{v \in [n]} \psi_A(v,  v \alpha)$.
\end{claim}
\begin{proof}
Let $\omega = (a_1 \to b_1) \dots (a_l \to b_l)$ be a shortest word expressing $\alpha$ in $A^*$, with $l = \ell(A, \alpha)$. Say that the arc $(a_i \to b_i)$, $i \geq 2$, \emph{carries} $v \in [n]$ if $v (a_1 \to b_1) \dots (a_{i-1} \to b_{i-1}) = a_i$ (assume that $a_1 \to b_1$ only carries $a_1$). Every arc $(a_i \to b_i)$ carries at least one vertex, for otherwise we could remove that arc form the word $\omega$ and obtain a shorter word still expressing $\alpha$. Let $v \in [n]$, and denote $v_0 = v$ and $v_i = v (a_1 \to b_1) \dots (a_i \to b_i)$ (and hence $v_l = v \alpha$). Let us remove the repetitions in this sequence: let $j_0 = 0$ and for $i \ge 1$, $j_i = \min\{j : v_j \ne v_{j_{i-1}}\}$. Then the sequence $v = v_{j_0}, v_{j_1}, \dots, v_{j_{l(v)}} = v \alpha$ forms a path in $A$ of length $l(v)$, and hence $l(v) \le \psi(v, v\alpha)$. For each $v \in [n]$, there are $l(v)$ arcs in $\omega$ carrying $v$, so the length of $\omega$ satisfies
\[ 	 l \le \sum_{v =1}^n l(v) \le \sum_{v \in [n]} \psi_A(v,  v \alpha). \]
\end{proof}

\begin{claim}\label{lem:psi(A)1}
If $|L| \ge 2$, then $\sum_{v \in [n]} \psi_A(v,  v \alpha) \le \frac{(n-r)(n+r-3)}{2}$.
\end{claim}
\begin{proof}
As $|L| \ge 2$, and $A$ is topologically sorted, we have $\{ n, n-1 \} \subseteq  L$, and any $\alpha \in \genset{A}$ fixes both $n-1$ and $n$, i.e. $\psi_A(v,  v \alpha) = 0$ for $v \in \{n-1, n\}$. For any $v \in [n-2]$, we have 
\[ 	\psi_A(v, v\alpha) \le \min\{ n-1, v\alpha\} - v. \]
Hence
\begin{align*}
	\sum_{v \in [n]} \psi_A(v,  v \alpha) &= \sum_{v \in [n-2]} \psi_A(v,  v \alpha)\\
	&\le \sum_{v \in [n-2]} \left( \min\{ n-1, v\alpha\} - v \right)\\
	&= \sum_{w \in [n-2]\alpha} \left( \min\{ n-1, w\} |w \alpha^{-1}| \right) - T_{n-2},
\end{align*}
where $T_k = \frac{k(k+1)}{2}$. The summation is maximised when $|n \alpha^{-1}| = n-r$ and $|w \alpha^{-1}| = 1$ for $n-r+1 \le w \le n-2$, thus yielding
\begin{align*}
	\sum_{v \in [n]} \psi_A(v,  v \alpha) &\le (n-1)(n-r) + (T_{n-2} - T_{n-r}) - T_{n-2}\\
	&= \frac{(n-r)(n+r-3)}{2}.
\end{align*}
\end{proof}

\begin{claim}\label{lem:psi(A)2}
If $\vert L \vert = 1$, then $\ell(A, \alpha) \le \frac{(n-r)(n+r-3)}{2} + 1$.
\end{claim}
\begin{proof}
As $A$ is topologically sorted, $L = \{ n\}$. We use the notation from the proof of Claim \ref{lem:l_A}. We then have $l(n) = 0$. We have three cases:
\begin{description}
\item[Case 1:] $(n-1)$ is fixed by $\alpha$. Then, $l(n-1) = 0$ and $l(v) \le \min\{n-1, v \alpha\} - v$ for all $v \in [n-2]$. By the same reasoning as in Claim \ref{lem:psi(A)1}, we obtain $\ell(A, \alpha) \le \frac{(n-r)(n+r-3)}{2}$. 

\item[Case 2:] $(n-1) \alpha = n$ and $v \alpha \le n-1$ for every $v \in [n-2]$. Then again $l(v)\le \min\{n-1, v\alpha \} -v$, for all $v \in [n-2]$, and $\ell(A, \alpha) \le \frac{(n-r)(n+r-3)}{2}$. 

\item[Case 3:] $n$ has at least two pre-images under $\alpha$. Let $\omega = (a_1 \to b_1) \dots (a_l \to b_l)$ be a shortest word expressing $\alpha$ in $A^*$, and denote $\alpha_0 = \id$ and $\epsilon_i = (a_i \to b_i)$, $\alpha_i = \epsilon_1 \dots \epsilon_i$ for $i \in [l]$. We partition $n \alpha^{-1}$ into two parts $S$ and $T$: 
\[ 	S = \{v \in n \alpha^{-1} : v_{l(v) - 1} = n-1\}, \quad T = n \alpha^{-1} \setminus S. \] 
For all $v \in S$, if the arc carrying $v$ to $n-1$ is $\epsilon_j$, then $(n-1)\alpha_{j-1}^{-1} \subseteq S$ ($v$ can only collapse with other pre-images of $\alpha$). Then the arc $(n-1 \to n)$ occurs only once in the word $\omega$ (if it occurs multiple times, then remove all but the last occurrence of that arc to obtain a shorter word expressing $\alpha$). If we do not count that arc, we have $l'(v) \le n-1-v$ arcs carrying $v$ if $v \in S$, $l(v) \le n-1-v$ arcs carrying $v$ if $v \in T$, and $l(v) \le v\alpha - v$ if $v\alpha \ne n$. Again, we obtain $\ell(A, \alpha) \le \frac{(n-r)(n+r-3)}{2} + 1$.
\end{description}
\end{proof}
\end{proof}

The following lemma completes the proof of Theorem \ref{th:lAcy}.

\begin{lemma}
For any $n \geq 3$ and $r \in [n-1] \setminus \{ 1\}$, there exists an acyclic digraph $Q_n$ on $[n]$ and a transformation $\beta_r \in \genset{Q_n}$ of rank $r$ such that
\[ \ell(Q_n, \beta_r) \geq  \frac{(n-r)(n+r-3)}{2} + 1. \]
\end{lemma}
\begin{proof}
Let $Q_n$ be the acyclic digraph on $[n]$ with edge set
\[ E(Q_n) := \left\{ (u, u+1) : u \in [n-1] \right\} \cup \left\{ (n-2, n) \right\}. \]
For any $r \in [n-1] \setminus \{ 1\}$, define $\beta_r \in \genset{Q_n}$ by
\[	v\beta_r := \begin{cases}
		n-r+v &\text{if } v \in [r-2], \\
		n - 1 &\text{if } v \in [n-1] \setminus [r-2], \  n-v \equiv 0 \mod 2,\\
		n &\text{if }v \in [n-1] \setminus [r-2], \  n-v \equiv 1 \mod 2,\\
		n &\text{if } v = n.
	\end{cases} \]

Let $\beta_r$ be expressed as a word in $Q_n^*$ of minimum length as
\[ 	\beta_r = (a_1 \to b_1) \dots (a_l \to b_l), \]
where $l= \ell(Q_n, \beta_r)$. Denote $\alpha_0 := \id$, $\epsilon_i := (a_i \to b_i)$, and $\alpha_i := \epsilon_1 \dots \epsilon_i$, for $i \in [l]$. Say that $\epsilon_i$ carries $u \in [n]$ if $u \alpha_{i-1} = a_i$ and hence $u \alpha_i \ne u \alpha_{i-1}$.

\begin{claim}
For each $i \in [l]$, the arc $\epsilon_i$ carries exactly one vertex.
\end{claim}
\begin{proof}
First, $(a_1, b_1) \in E(Q_n)$ and $a_1 \beta_r = b_1 \beta_r$ imply that $a_1 = n-1$ and $b_1 = n$. Suppose that there is an arc $\epsilon_j$, $j \in [l]$, that carries two vertices $u < v$; take $j$ to be minimal index with this property. We remark that $v \le n-2$ and $u \alpha_{j-1} = v \alpha_{j-1}$ imply $u \beta_r = v \beta_r$. Then $w := u+1$ satisfies $w \beta_r \ne u \beta_r$, so $w$ is not carried by $\epsilon_j$. If $w \alpha_{j-1} \le n-2$, then $u \alpha_{j-1} < w \alpha_{j-1} < v \alpha_{j-1}$ since $u < w < v$ and the graph induced by $[n-2]$ in $Q_n$ is the directed path $\vec{P}_{n-2}$; this contradicts that $u \alpha_{j-1} = v \alpha_{j-1}$. Hence $w \alpha_{j-1} \ge n-1$ and $v \alpha_{j-1} \ge n-1$. If $v \alpha_{j-1} = n$ or $v \beta_r = n-1$, then $\epsilon_j$ does not carry $v$. Thus, $v\alpha_{j-1} = n-1$ and $v \beta_r = n$. Then,  in order to carry $v$ to $n-1$, we have $\epsilon_s = (n-2 \to n-1)$ for at least one $s \in [l]$, and $\epsilon_j = (n-1 \to n)$. For $s \in [j-1]$, replace all occurrences $\epsilon_s = (n-2 \to n-1)$ with $\epsilon_s^\prime := (n-2 \to n)$ and delete $\epsilon_j$: this yields a word in $Q_n^*$ of length $l^\prime<l$ expressing $\beta_r$, which is a contradiction.
\end{proof}

For all $ i \in [l]$, denote $\delta(i) := \sum_{v \in [n]} d_{Q_n}(v \alpha_i,  v \beta_r)$. We then have $\delta(l) = 0$, and by the claim, $\delta(i) \ge \delta(i-1) - 1$ for all $i \in [l]$. Thus $l \ge \delta(0)$, where
\begin{align*}
	\delta(0) &= \sum_{v \in [n]} d_{Q_n}(v,  v \beta_r)\\
	&= \sum_{v=1}^{r-2} (n-r) + \sum_{v=r-1}^{n-2} (n-1 - v) + 1\\
	&= \frac{(n-r)(n+r-3)}{2} + 1.
\end{align*}
\end{proof}


\subsection{Strong tournaments}

Let $n \geq 3$. Recall that if $T$ is a strong tournament on $[n]$, then $\{ a \to b : (a,b) \in E(T) \}$ is a minimal generating set of $\Sing_n$. Let $\Tour_n$ denote the set of all strong tournaments on $[n]$. For $r \in [n-1]$, define
\begin{align*}
	\ell_{\max}^{\Tour}(n, r) &:= \max \{ \ell(T, r) : T \in \Tour_n  \},\\
	\ell_{\max}^{\Tour}(n) &:= \max \{ \ell(T) : T \in \Tour_n  \}.
\end{align*}

Define analogously $\ell_{\min}^{\Tour}(n,r)$ and $\ell_{\min}^{\Tour}(n)$. The first few values of $\ell_{\min}^{\Tour}(n,r)$ and $\ell_{\max}^{\Tour}(n,r)$, calculated with the GAP package \emph{Semigroups} \cite{M15} using data from \cite{MK15}, are given by Table \ref{table:lmax_tour}. The calculation of these values has been the inspiration for the results and conjectures of this section. 

\begin{table}[!hb]
\begin{center}
\begin{tabular}{c|c|c|c|c|c}
	$n \backslash r$ & $2$ & $3$ & $4$ & $5$ & $6$ \\
	\hline 
	$3$ & $(6, 6)$ & & & & \\
	$4$ & $(8,8)$ & $(11,11)$ & & & \\
	$5$ & $(6, 11)$ & $(8, 14)$ & $(10, 17)$ & & \\
	$6$ & $(8, 13)$ & $(10, 18)$ & $(11, 21)$ & $(13, 24)$ & \\
	$7$ & $(8, 16)$ & $(10, 22)$ & $(11, 26)$ & $(13, 29)$ & $(15, 32)$ 
\end{tabular}
\caption{First values of $\left( \ell_{\min}^{\Tour}(n,r) , \ell_{\max}^{\Tour}(n,r) \right)$.} \label{table:lmax_tour}
\end{center}
\end{table}

\begin{lemma}\label{lem:same_kernel} 
Let $n \geq 3$ and $T \in \Tour_n$.
\begin{enumerate}
	\item For any partition $P$ of $[n]$ into $r$ parts, there exists an idempotent $\alpha \in \Sing_n$ with $\ker(\alpha) = P$ such that $\ell(T, \alpha)= n - r$.

	\item For any $r$-subset $S$ of $[n]$, there exists an idempotent $\alpha \in \Sing_n$ with $\Ima(\alpha) = S$ such that $\ell(T, \alpha)= n - r$.
\end{enumerate}
\end{lemma}
\begin{proof}
\begin{enumerate}
\item Let $P = \{ P_1, \dots, P_r \}$. For all $1 \le i \le r$, the digraph $T[P_i]$ induced by $P_i$ is a tournament, so it is connected and there exists a vertex $v_i$ reachable by any other vertex in $P_i$: let $\alpha$ map the whole of $P_i$ to $v_i$. Then $\alpha$, when restricted to $P_i$, is a constant map, which can be computed using $|P_i| - 1$ arcs. Summing for $i$ from $1$ to $r$, we obtain that $\ell(T, \alpha) = n - r$.

\item Without loss of generality, let $S = [r] \subseteq [n]$. For every $v \in [n]$, define
\[ 	s(v) := \min \{ s \in S: d_T(s', v) \geq d_T(s,v), \forall s' \in S\}. \]
In particular, if $v \in S$, then $s(v) = v$. Moreover, if $v = v_0, v_1, \dots, v_d = s(v)$ is a shortest path from $v$ to $s(v)$, with $d = d_T(v,s(v))$, then $s(v_i) = s(v)$ for all $0 \le i \le d$. For each $v \in [n ]$, fix a shortest path $P_v$ from $v$ to $s(v)$, and consider the digraph $D$ on $[n]$ with edges 
\[ E(D) := \{ (a,b) : (a,b) \in E(P_v) \text{ for some } v \in [n]\}. \]
Then, $D$ is acyclic and the set of vertices with out-degree zero in $D$ is exactly $S$. Let sort $[n]$ so that $D$ has reverse topological order: $(a,b) \in E(D)$ only if $a >b$. Note that $S$ is fixed by this sorting. Let $\alpha$ be given by $v \alpha:= s(v)$; hence, with the above sorting 
\[  \alpha = \bigcirc_{v=n}^{r+1} (v \to v_1). \]
\end{enumerate}
\end{proof}

\begin{lemma} \label{le:idempotents}
Let $n \geq 3$, $T \in \Tour_n$, and $\alpha:=(u \to v) \in \Sing_n$, for $(u,v) \not \in E(T)$. Then 
\[ \ell(T, \alpha) = 4 d_T(u,v) - 2. \]
\end{lemma}
\begin{proof}
Let $u = v_0, v_1, \dots, v_d = v$ be a shortest path from $u$ to $v$ in $T$, where $d := d_T(u,v)$. As $(u,v) \not \in E(T)$ and $T$ is a tournament, we must have $(v,u) \in E(T)$. By the minimality of the path, for any $j +1 < i$, we have $(v_j, v_i) \not \in E(T)$, so $(v_i, v_j) \in E(T)$. Then, the following expresses $\alpha$ with arcs in $T^*$: 
\begin{align*}
	(v_0 \to v_d)  = & (v_d \to v_0) (v_{d-1}  \to v_d) (v_{d-2} \to v_{d-1}) \cdots (v_1 \to  v_2) (v_0 \to v_1 ) \\
	& \left( (v_2 \to v_0) (v_1  \to v_2) \right) \left( (v_3 \to v_1) (v_2 \to v_3) \right) \cdots \left( (v_d \to v_{d-2}) (v_{d-1} \to v_d) \right) \\
	& (v_{d-2} \to v_{d-1}) \cdots (v_0 \to v_1).
\end{align*}

So $\ell(T, \alpha) \leq 4d-2$. For the lower bound, we note that any word in $T^*$ expressing $(u \to v)$ must begin with $(v \to u)$. Then, $u$ has to follow a walk in $T$ towards $v$; say this walk has length $l \ge d$. All the vertices on the walk must be moved away (as otherwise they would collapse with $u$) and have to come back to their original position (since $\alpha$ fixes them all); as the shortest cycle in a tournament has length $3$, this process adds at least $3(l-1)$ symbols to the word. Altogether, this yields a word of length at least
\[ 	1 + l + 3(l-1) = 4l-2 \ge 4 d - 2. \]
\end{proof}

Let $n = 2m+1 \geq 3$ be odd, and let $\kappa_n$ be the \emph{circulant tournament} on $[n]$ with edges $E(\kappa_n):=\{ (i, (i+j) \mod n): i \in [n], j \in [m] \}$. Figure \ref{fig:kappa5} illustrates $\kappa_5$. In the following theorem, we use $\kappa_n$ to provide upper and lower bounds for $\ell_{\min}^{\Tour}(n,r)$ and $\ell_{\max}^{\Tour}(n,r)$ when $n$ is odd. 

\begin{figure}[!h]
\begin{center}
\begin{tikzpicture}

	\node (1) at (90-72*1:2) {1};
	\node (2) at (90-72*2:2) {2};
	\node (3) at (90-72*3:2) {3};
	\node (4) at (90-72*4:2) {4};
	\node (5) at (90-72*5:2) {5};	
	
    \draw[-latex] (1) -- (2);
    \draw[-latex] (2) -- (3);
    \draw[-latex] (3) -- (4);
    \draw[-latex] (4) -- (5);
    \draw[-latex] (5) -- (1);

	\draw[-latex] (1) -- (3);
    \draw[-latex] (2) -- (4);
    \draw[-latex] (3) -- (5);
    \draw[-latex] (4) -- (1);
    \draw[-latex] (5) -- (2);

\end{tikzpicture}
\end{center}
\caption{The circulant tournament $\kappa_5$.} \label{fig:kappa5}
\end{figure}
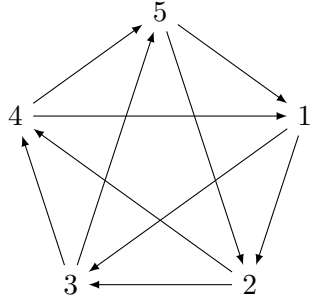

\begin{theorem}
For any $n$ odd, we have
\begin{align*}
	n + r - 2 &\le \ell_{\min}^{\Tour}(n,r) \le n + 8r,\\
	(\hat{r} + 1)(n - \hat{r}) - 1 &\le \ell_{\max}^{\Tour}(n,r) \le 6rn + n -10r.
\end{align*}
where $\hat{r} = \min\{r-1, \lfloor n/2 \rfloor\}$.
\end{theorem}
\begin{proof}
Let  $T \in \Tour_n$ and $2 \le r \le n-1$. We introduce the following notation:
\begin{align*}
 	[n]_r & := \{ {\bf u} := (u_1, \dots, u_r) : u_i \neq u_j, \forall i, j \},  \\
	\Delta(T, r) &:= \max \left\{  \sum_{i=1}^r d_T(u_i, v_i) : {\bf u}, {\bf v} \in [n]_r \right\}.
\end{align*}
The result follows by the next claims.
\begin{claim}
$r'(\diam(T) - r' + 1) + r-r' \le \Delta(T,r) \le r \diam(T)$, where $r' = \min\{r, \lfloor ( \diam(T)+1)/2 \rfloor\}$. 
\end{claim}
\begin{proof}
The upper bound is clear. For the lower bound, let $u, v \in [n]$ be such that $d_T(u,v) = \diam(T)$, and let $u = v_0, v_1, \dots, v_d = v$ be a shortest path from $u$ to $v$, where $d = \diam(T)$. Then, $d_T(v_i,v_j) = j-i$, for all $0 \le i \le j \le D$. If $1 \le r \le \lfloor (d+1)/2 \rfloor$, consider ${\bf u}' = (v_0, \dots, v_{r-1})$ and ${\bf v}' = (v_{d-r+1}, \dots, v_d)$, so we obtain $\Delta(T, r) \ge r(d-r+1)$. If $r \ge\lfloor (d+1)/2 \rfloor$, simply add vertices $u'_j$ and $v'_j$ such that $(u'_j, v'_j) \notin T$.
\end{proof}

\begin{claim}
 $\min \{ \Delta(T,r) : T \in \Tour(n) \} = \Delta(\kappa_n,r) = 2r$.
\end{claim}
\begin{proof}
Let ${\bf u} = (u_1, \dots, u_n)$ form a Hamiltonian cycle, and choose ${\bf v} = (u_n, u_1, \dots, u_{n-1})$. Then $d_T(u_i, v_i) \ge 2$ for all $i$. Conversely, since $\diam(\kappa_n) = 2$, we have $\Delta(\kappa_n,r) = 2r$.
\end{proof}

\begin{claim}
$n-r + \Delta(T, r-1) \le \ell(T,r) \le n + 6 r \diam(T) - 4r$.
\end{claim}
\begin{proof}
For the lower bound, consider $\alpha \in \Sing_n$ as follows. Let ${\bf u} = (u_1, \dots, u_{r-1})$ and ${\bf v} = (v_1, \dots, v_{r-1})$ achieve $\Delta(T,r-1)$, and let $v \notin \{v_1, \dots, v_{r-1}\}$; define 
\[ 	x \alpha = \begin{cases}
		v_i &\text{if } x = u_i, \\
		v &\text{otherwise}.
	\end{cases} \]
Let $\omega = e_1 \dots e_l$ (where $e_i =(a_i \to b_i)$) be a shortest word expressing $\alpha$, where $l := \ell(T,\alpha)$. Recall that an arc $e_i$ carries a vertex $c$ if $c e_1 \dots e_{i-1} = a_i$. By the minimality of $\omega$, every arc carries at least one vertex. Moreover, if $c$ and $d$ are carried by $e_i$, then $c \alpha = d \alpha$; therefore, we can label every arc $e_i$ of $\omega$ by an element $c(e_i) \in \Ima(\alpha)$ if $e_i$ carries vertices eventually mapping to $c(e_i)$. Denote the number of arcs labelled $c$ as $l(c)$, we then have $l = \sum_{c \in \Ima(\alpha)} l(c)$. For any $u \in V$, there are at least $d_T(u, u \alpha)$ arcs carrying $u$. Therefore,
\[ 	l = \sum_{c \in \Ima(\alpha)} l(c)
	\ge \sum_{i=1}^{r-1} d_T(u_i, v_i) + \sum_{a \notin {\bf u}} d_T(a, v)
	\ge \Delta(T, r-1) + n-r.\]

For the upper bound, we can express any $\alpha \in \Sing_n$ of rank $r$ in the following fashion. By Lemma \ref{lem:same_kernel}, there exists $\beta \in \Sing_n$ with the same kernel as $\alpha$ such that $\ell(T, \beta) = n-r$. Suppose that $\Ima(\alpha) = \{ v_1, \dots, v_r \}$ and $\Ima(\beta) = \{u_1, \dots, u_r \}$, where $u_i \beta^{-1}  = v_i \alpha^{-1}$, for $i \in [r]$. Let $h \in [n] \setminus \Ima(\beta)$. Define a transformation $\gamma$ of $[n]$ by 
\[ 	x \gamma = \begin{cases}
		v_i &\text{if } x = u_i,\\ 
		 v_1 & \text{if } x= h, \\
		x &\text{otherwise}.
	\end{cases}\]
Then $\alpha = \beta \gamma$, where $\gamma \in \Sing_n$, and by Theorem \ref{Howie-Iwahori}
\[ 	\ell(K_n, \gamma) = n - \fix(\gamma) + \cycl(\gamma) \le r + \frac{r}{2} = \frac{3r}{2}. \]
By Lemma \ref{le:idempotents}, each arc associated to $K_n$ may be expressed in at most $4 \diam(T) - 2$ arcs associated to $T$; therefore, 
\[ 	\ell(T, \gamma) \leq \frac{3r}{2}( 4 \diam(T) - 2) = 6r \diam(T) - 3r. \]
Thus,
\[ \ell(T,\alpha) \leq \ell(T,\beta) + \ell(T,\gamma) \leq n + 6r\diam(T) -4r. \]
\end{proof}
\end{proof}


We finish this section by proposing two conjectures. Let $\pi_n$ be the tournament on $[n]$ with edges $E(\pi_n):=\{ ( i, (i+1) \mod n ) : i \in [n]\} \cup \{ (i, j) : j +1 < i \}$. Figure \ref{fig:pi5} illustrares $\pi_5$. 

\begin{figure}[!ht]
\begin{center}
\begin{tikzpicture}
	\node (1) at (0,0) {1};
	\node (2) at (2,0) {2};
	\node (3) at (4,0) {3};
	\node (4) at (6,0) {4};
	\node (5) at (8,0) {5};
	
	\draw[-latex] (1) -- (2); 
	\draw[-latex] (2) -- (3);
	\draw[-latex] (3) -- (4);
	\draw[-latex] (4) -- (5);
	
	\draw[-latex] (3) .. controls(2,1) .. (1);
	\draw[-latex] (4) .. controls(3,2) .. (1);
	\draw[-latex] (5) .. controls(4,3) .. (1);
	
	\draw[-latex] (4) .. controls(4,1) .. (2);
	\draw[-latex] (5) .. controls(5,2) .. (2);
	
	\draw[-latex] (5) .. controls(6,1) .. (3);
\end{tikzpicture}
\end{center}
\caption{$\pi_5$} \label{fig:pi5}
\end{figure}

\begin{conjecture}\label{conj}
For every $n \geq 3$, $r \in [n-1]$, and $T \in \Tour_n$, we have
\[  \ell(T, r) \leq \ell(\pi_n, r) = \ell_{\max}^{\Tour}(n, r) , 
\]
with equality if and only if $T \cong \pi_n$. Furthermore,
\[ \ell(\pi_n) = \ell_{\max}^{\Tour}(n) = \frac{n^2 + 3n - 6}{2},\]
which is achieved for $\alpha : = n \ (n-1) \ \dots \ 2 \ n$.
\end{conjecture}

Tournament $\pi_n$ has appeared in the literature before: it is shown in \cite{M66} that $\pi_n$ has the minimum number of strong subtournaments among all strong tournaments on $[n]$. On the other hand, it was shown in \cite{BH65} that, for $n$ odd, the circulant tournament $\kappa_n$ has the maximal number of strong subtournaments among all strong tournaments on $[n]$. 

\begin{conjecture}
For every $n \geq 3$ odd, $r \in [n-1]$, and $T \in \Tour_n$, we have  
\[  \ell_{\min}^{\Tour}(n, r) = \ell(\kappa_n, r).\]
Furthermore, 
\[ 	\ell_{\min}^{\Tour}(n,2) = n+1 \ \text{ and } \ \ell_{\min}^{\Tour}(n,r) = n+r, \]
for all $3 \le r \le \frac{n+1}{2}$.
\end{conjecture}


\end{document}